\documentclass[11pt,twoside]{article}
\usepackage[T1]{fontenc}
\usepackage{textcomp}
\usepackage{mathrsfs}
\usepackage{amsmath}
\usepackage{amssymb}
\usepackage{fancyhdr}
\usepackage{latexsym}
\usepackage{bbding}
\usepackage{mathrsfs}
\usepackage{wasysym}
\usepackage{cite}
\usepackage{multicol,graphics}
\usepackage{xcolor}

\RequirePackage{times}

\def\XXint#1#2#3{{\setbox0=\hbox{$#1{#2#3}{\int}$ }
		\vcenter{\hbox{$#2#3$ }}\kern-.6\wd0}}

\newtheorem{theorem}{Theorem}[section]
\newtheorem{lemma}[theorem]{Lemma}

\newtheorem{definition}[theorem]{Definition}

\newtheorem{proposition}[theorem]{Proposition}
\numberwithin{equation}{section}
\newenvironment{proof}[1][Proof]{\noindent\textbf{#1.} }{\hfill $\Box$}

\allowdisplaybreaks[4]
\makeatletter
\begin{document}
	
\title{Well-posedness of the 2D surface quasi-geostrophic equation in variable Lebesgue spaces\footnote{E-mail: haochen990118@163.com, gaston.vergarahermosilla@univ-evry.fr, jihzhao@163.com.}}
\author{Hao Chen$^{\text{1}}$, Gast\'on Vergara-Hermosilla$^{\text{2}}$, Jihong Zhao$^{\text{1}}$\\
[0.2cm] {\small $^{\text{1}}$School of Mathematics and Information Science, Baoji University of Arts and Sciences,}\\
[0.2cm] {\small  Baoji, Shaanxi 721013,  China}\\
[0.2cm] {\small $^{\text{2}}$LaMME, Univ. Evry, CNRS, Universit\'e Paris-Saclay, 91025, Evry, France}}

\date{\today}
\maketitle
\begin{abstract}
In this paper, we are mainly concerned with the well-posedness of the dissipative surface quasi-geostrophic equation in the framework of  variable Lebesgue spaces. Based on some analytical results developed in the variable Lebesgue spaces and the $L^{p}$-$L^{q}$ decay estimates of the fractional heat kernel, we establish the local existence and regularity of solutions to the 2D dissipative surface quasi-geostrophic equation in the variable Lebesgue space.
\end{abstract}
\smallbreak

\smallskip


\section{Introduction}
In this paper, we study the Cauchy problem of the dissipative surface quasi-geostrophic
equation in two dimensions. The equation is derived from the more general quasi-geostrophic approximation for a rapidly rotating fluid flow with small Rossby and Ekman numbers, which describes the potential temperature dynamics of atmospheric and ocean flow (cf. \cite{CMT94, P87}). The equation is given by
\begin{equation}\label{eq1.1}
\begin{cases}
  \partial_{t}\theta+u\cdot\nabla\theta+\mu\Lambda^{\alpha}\theta=f,\quad &t>0,\ x\in\mathbb{R}^{2},\\
   \theta(0,x)=\theta_{0}(x),\quad &x\in\mathbb{R}^{2},
\end{cases}
\end{equation} 
where $\theta_{0}$ is an initial data, $f$ is an external force, $\alpha\in(0, 2]$, and  $\mu>0$ is the dissipative coefficient.  The scalar function $\theta$ represents the potential temperature of the fluid, and the velocity field
$u=(u_{1},u_{2})$ of the fluid is determined by $\theta$ as
 \begin{equation}\label{eq1.2}
u=(u_{1},u_{2})=(-\partial_{x_{2}}\Lambda^{-1}\theta, \partial_{x_{1}}\Lambda^{-1}\theta)=(-\mathcal{R}_{2}\theta, \mathcal{R}_{1}\theta).
\end{equation}
The pseudo-differential operators $\Lambda^{\alpha}$ and $\mathcal{R}_{j}$ ($j=1,2$) denote the fractional Laplacian and the Riesz transforms on $\mathbb{R}^{2}$, respectively, which are defined by
 \begin{equation*}
 \mathcal{F}(\Lambda^{\alpha}\theta)(\xi):=|\xi|^{\alpha}\mathcal{F}(\theta)(\xi),  \ \ \ \mathcal{F}(\mathcal{R}_{j}\theta)(\xi):=\frac{i\xi_{j}}{|\xi_{j}|}\mathcal{F}(\theta)(\xi),
\end{equation*}
where $\mathcal{F}(\theta)$ denotes the Fourier transform of $\theta$. Clearly, by \eqref{eq1.2}, $u$ is divergence free.
\smallbreak

Based on the framework of the Kato's analytical semigroup, we can rewrite the equation \eqref{eq1.1} as an equivalent integral form:
\begin{equation}\label{eq1.3}
  \theta(t)=G^{\alpha}_{t}\ast \theta_{0}(x)+\int_{0}^{t}G^{\alpha}_{t-s}\ast f(s,x)ds-\int_{0}^{t}G^{\alpha}_{t-s}\ast(u\cdot\nabla\theta)(s,x)ds,
\end{equation}
where $G^{\alpha}_{t}(x)$ is the fractional heat kernel. Any solution satisfying the integral equation \eqref{eq1.3} is called a \textit{mild solution} of \eqref{eq1.1}.  An universal approach to find such a mild solution is to apply the Banach-Picard principle for the mapping
\begin{equation*}
  \mathfrak{F}(\theta)(t,x)=G^{\alpha}_{t}\ast \theta_{0}(x)+\int_{0}^{t}G^{\alpha}_{t-s}\ast f(s,x)ds-\int_{0}^{t}G^{\alpha}_{t-s}\ast(u\cdot\nabla\theta)(s,x)ds
\end{equation*}
 in various critical function spaces.
 \smallbreak

 Let us recall that the critical function spaces for initial data of the equation \eqref{eq1.1} should be invariant under the scaling transform
\begin{align}\label{eq1.4}
\theta_{0,\lambda}(x):=\lambda^{\alpha-1}\theta_{0}(\lambda x)
\end{align}
for any $\lambda>0$. Based on the scaling \eqref{eq1.4}, one can easily check that the Lebesgue space $L^{\frac{2}{\alpha-1}}(\mathbb{R}^{2})$, the weak Lebesgue space $L^{\frac{2}{\alpha-1},\infty}(\mathbb{R}^{2})$, the Sobolev space $\dot{H}^{2-\alpha}(\mathbb{R}^{2})$, the Besov space $\dot{B}^{1-\alpha+\frac{2}{p}}_{p,q}(\mathbb{R}^{2})$ and the Triebel-Lizorkin space $\dot{F}^{1-\alpha+\frac{2}{p}}_{p,q}(\mathbb{R}^{2})$ are critical spaces of the equation \eqref{eq1.1}. The well-posedness of the equation \eqref{eq1.1} in these critical spaces has been extensively studied in the last two decades,  see \cite{CF07, CF08} in the Lebesgue space $L^{\frac{2}{\alpha-1}}(\mathbb{R}^{2})$ and the weak Lebesgue space $L^{\frac{2}{\alpha-1},\infty}(\mathbb{R}^{2})$;  \cite{J04, M06} in the Sobolev space $\dot{H}^{2-\alpha}(\mathbb{R}^{2})$; \cite{AH08, CL03, CMZ07, DL10, W06} in the Besov space  $\dot{B}^{1-\alpha+\frac{2}{p}}_{p,q}(\mathbb{R}^{2})$ and \cite{CZ07} in the Triebel-Lizorkin space $\dot{F}^{1-\alpha+\frac{2}{p}}_{p,q}(\mathbb{R}^{2})$.
 For more global and local well-posedness results, we refer the readers to see \cite{B08, BB15, BBT13, CCW01, DD08,DL10, HK07, L13, W07} and references cited there.
\smallbreak

The main purpose of this paper is to explore some existence results for the equation \eqref{eq1.1} in the framework of the Lebesgue spaces $L^{p(\cdot)}(\mathbb{R}^{2})$ with a variable exponent $p(\cdot)$, which-to the best of our knowledge-were
not used before in the analysis of the dissipative quasi-geostrophic equation \eqref{eq1.1}. These spaces are quite
different from the usual Lebesgue spaces $L^{p}(\mathbb{R}^{2})$. Indeed, the variable Lebesgue spaces are a generalization of the classical Lebesgue spaces, replacing the constant exponent $p$ with a variable exponent function $p(\cdot)$.  Such spaces allow some functions with singularity at some point $x_{0}$, which is not in $L^{p}$ for any $1\leq p\leq \infty$ due to it either grows too quickly at $x_{0}$ or decays too slowly at infinity.  The main idea of our results is inspired by papers \cite{CV23, V24}, where the well-posedness of the 3D incompressible (fractional) Navier-Stokes equations have been established in the variable Lebesgue spaces.
\smallbreak

Now we state the main results of this work.  For simplicity, we assume that $\mu=1$ throughout this paper. The first result is the local existence of solutions with initial data $\theta_{0}\in L^{\bar{p}(\cdot)}(\mathbb{R}^{2})$ and $f\in L^{1}(0,\infty; L^{\bar{p}(\cdot)}(\mathbb{R}^{2}))$ (see Section 2 for the definition of these spaces).

\begin{theorem}\label{th1.1}
Let $1<\alpha\leq2$, $q(\cdot)\in \mathcal{P}^{\log}(0,+\infty)$ with $2<q^{-}\leq q(\cdot)\leq q^{+}<+\infty$, fix an index $p>\frac{2}{\alpha-1}$ satisfying the relationship $\frac{\alpha}{q(\cdot)}+\frac{2}{p}<\alpha-1$, and $\bar{p}(\cdot)\in \mathcal{P}^{\text{emb}}_{p}(\mathbb{R}^{2})$. Then for any $\theta_{0}\in L^{\bar{p}(\cdot)}(\mathbb{R}^{2})$ and $f\in L^{1}\left(0,\infty; L^{\bar{p}(\cdot)}(\mathbb{R}^{2})\right)$, there exists a time $T>0$ such that the equation \eqref{eq1.1} admits an unique local solution $\theta$ in the space $L^{q(\cdot)}\left(0,T; L^{p}(\mathbb{R}^{2})\right)$.
\end{theorem}

In our second result, we show that the solutions constructed in Theorem \ref{th1.1} are in fact regular by considering suitable assumptions on initial data $\theta_{0}$ and external force $f$.

\begin{theorem}\label{th1.2}
Let the assumptions of Theorem \ref{th1.1} be in force. Assume further that for all non-negative multi-index $\beta$ such that
\begin{equation}\label{eq1.5}
    D^\beta \theta_0 \in L^p\left(\mathbb{R}^2\right) \quad \text { and } \quad D^\beta f \in L^{1}\left(0, T ; L^{p}\left(\mathbb{R}^2\right)\right).
\end{equation}
Then, for any non-negative multi-index $\gamma$ such that $|\gamma|\leq |\beta|$, we have
\begin{equation}\label{eq1.6}
    D^\gamma \theta \in L^{q(\cdot)}\left(0, T; L^{p}\left(\mathbb{R}^2\right)\right).
\end{equation}
\end{theorem}

Let us remark here that when we address the problem of global existence of solutions to the equation \eqref{eq1.1}, as the second author of this paper suggested in \cite{V24}, we introduce a mixed variable Lebesgue space $\mathcal{L}^{p(\cdot)}_{\frac{2}{\alpha-1}}(\mathbb{R}^{2})$, which can be regarded as the intersection of the spaces $L^{p(\cdot)}(\mathbb{R}^{2})$ and $L^{\frac{2}{\alpha-1}}(\mathbb{R}^{2})$. In this case, the functional space should be chosen as $\mathcal{L}^{p(\cdot)}_{\frac{2}{\alpha-1}}\left(\mathbb{R}^{2},L^{\infty}(0,T)\right)$, and we can establish the following bilinear estimate as
\begin{align}\label{eq1.7}
  \left\|\int_{0}^{t}G^{\alpha}_{t-s}\ast(u\cdot\nabla\theta)(s,x)ds\right\|_{L^{p(\cdot)}_{x}(L^{\infty}_{t})}
  &\leq C\left\|u\right\|_{\mathcal{L}^{p(\cdot)}_{\frac{2}{\alpha-1}}(L^{\infty}_{t})}\left\|\theta\right\|_{\mathcal{L}^{p(\cdot)}_{\frac{2}{\alpha-1}}(L^{\infty}_{t})}.
  \end{align}
However, we can not use $\left\|\theta\right\|_{\mathcal{L}^{p(\cdot)}_{\frac{2}{\alpha-1}}(L^{\infty}_{t})}$ to control the term $\left\|u\right\|_{\mathcal{L}^{p(\cdot)}_{\frac{2}{\alpha-1}}(L^{\infty}_{t})}$ due to the lack of flexibility of the indices that intervene in the boundedness of the Riesz transforms, and this is the main crux that intervenes us to get the global existence of solutions. In our opinion, it is a challenging open problem to establish the global existence of solutions to the equation \eqref{eq1.1} in a variable Lebesgue spaces.
\smallbreak

The rest of this paper is organized as follows. In Section 2, we present the detailed review of the definitions and properties of the variable exponent Lebesgue spaces, moreover, all kinds of analytic estimates in terms of the bounded properties of singular integral operators in the variable Lebesgue spaces are listed. In Section 3, by using the decay property of the fractional heat kernel, we successfully establish the linear and bilinear estimates of solutions in each solution space, then complete the proof of Theorem \ref{th1.1} by the Banach-Picard principle. In Section 4, we complete the proof of Theorem \ref{th1.2}. Throughout this paper, we shall indistinctly use the notations $C$ and $C_{i}(i=1,2,\cdot\cdot\cdot)$ for generic harmless constants.

\section{ Preliminaries}
To keep this paper reasonably self-contained, several results and definitions on variable Lebesgue spaces and fractional heat kernels are recalled. We begin with a fundamental definition. Given a set $\Omega\subset\mathbb{R}^{n}$, let $\mathcal{P}(\Omega)$ be the set of all Lebesgue measurable functions $p(\cdot):\Omega\rightarrow[1,+\infty]$.  The elements of $\mathcal{P}(\Omega)$ are called exponent functions or simply exponents. For any $p(\cdot)\in \mathcal{P}(\Omega)$, we denote
$$
 p^{-}:=\operatorname{essinf}_{x\in\Omega}p(x), \ \  p^{+}:=\operatorname{esssup}_{x\in\Omega}p(x).
$$
Throughout this paper, we will always assume that $1<p^{-}\leq p^{+}<+\infty$.
\smallbreak

For the classical Lebesgue space $L^{p}(\Omega)$ ($1\leq p<\infty$), its norm is directly defined  by
\begin{equation}\label{eq2.1}
\|f\|_{L^{p}}:=\left(\int_{\Omega}|f(x)|^{p}dx\right)^{\frac{1}{p}}.
\end{equation}
Such a definition obviously fails since we cannot replace the constant exponent $\frac{1}{p}$ outside the integral by the exponent function $\frac{1}{p(\cdot)}$. The solution is a more subtle approach which is similar to that used to define the Luxemburg norm on the Orlicz spaces (cf. \cite{CF13, DHHR11}).

\begin{definition}\label{de2.1}
Given a set $\Omega\subset\mathbb{R}^{n}$ and $p(\cdot)\in \mathcal{P}(\Omega)$, for a measurable function $f$, we define
\begin{equation}\label{eq2.2}
   \|f\|_{L^{p(\cdot)}}:=\inf\left\{\lambda>0:  \rho_{p(\cdot)}\left(\frac{f}{\lambda}\right)\leq 1\right\},
\end{equation}
where the modular function $\rho_{p(\cdot)}$ associated with $p(\cdot)$ is given by the expression
\begin{equation*}
   \rho_{p(\cdot)}(f):=\int_{\Omega}|f(x)|^{p(x)}dx.
\end{equation*}
Moreover, if the set on the right-hand side of  \eqref{eq2.2} is empty then we define $\|f\|_{L^{p(\cdot)}}=\infty$.
\end{definition}

\noindent\textbf{Remark 2.1}
Note that if the exponent function $p(\cdot)$ is a constant, i.e. if $p(\cdot)=p\in [1,\infty)$, then we obtain the usual norm \eqref{eq2.1} via the modular function $ \rho_{p}$.
\smallbreak

\begin{definition}\label{de2.2}
Given a set $\Omega\subset\mathbb{R}^{n}$ and $p(\cdot)\in \mathcal{P}(\Omega)$, we define the variable exponent Lebesgue space $L^{p(\cdot)}(\Omega)$ to be the set of Lebesgue measurable functions $f$ such that   $\|f\|_{L^{p(\cdot)}}<+\infty$.
\end{definition}

\noindent\textbf{Remark 2.2}
Given $\Omega\subset\mathbb{R}^{n}$ and $p(\cdot)\in \mathcal{P}(\Omega)$, $L^{p(\cdot)}(\Omega)$ is a vector space, and the function $\|\cdot\|_{L^{p(\cdot)}}$ defines a norm on $L^{p(\cdot)}(\Omega)$, thus $L^{p(\cdot)}(\Omega)$ is a normed vector space. Actually, $L^{p(\cdot)}(\Omega)$ is a Banach space associated with the norm $\|\cdot\|_{L^{p(\cdot)}}$.
\smallbreak

Next, we collect some properties of the variable exponent Lebesgue spaces.  The first one is a generalization of the H\"{o}lder's inequality (cf. \cite{CF13}, Corollary 2.28; \cite{DHHR11}, Lemma 3.2.20).

\begin{lemma}\label{le2.3}
Given $\Omega\subset\mathbb{R}^{n}$ and exponent functions $p_{1}(\cdot), p_{2}(\cdot)\in \mathcal{P}(\Omega)$, define $p(\cdot)\in \mathcal{P}(\Omega)$ by $\frac{1}{p(x)}=\frac{1}{p_{1}(x)}+\frac{1}{p_{2}(x)}$. Then there exists a constant $C$ such that for all $f\in L^{p_{1}(\cdot)}(\Omega)$ and $g\in L^{p_{2}(\cdot)}(\Omega)$, we have $fg\in L^{p(\cdot)}(\Omega)$ and
\begin{equation}\label{eq2.3}
   \|fg\|_{L^{p(\cdot)}}\leq C \|f\|_{L^{p_{1}(\cdot)}}\|g\|_{L^{p_{2}(\cdot)}}.
\end{equation}
\end{lemma}

Note that in the classical Lebesgue $L^{p}(\Omega)$ ($1\leq p<+\infty$), the norm can be computed by using the norm conjugate formula
 \begin{equation}\label{eq2.4}
   \|f\|_{L^{p}}\leq \sup_{\|g\|_{L^{p'}}\leq1}\int_{\Omega}|f(x)g(x)|dx.
\end{equation}
 A slightly weaker analog of the equality \eqref{eq2.4} is true for the variable Lebesgue spaces. Indeed, given $p(\cdot)\in \mathcal{P}(\Omega)$, we  define the conjugate exponent function by the formula
$$
  \frac{1}{p(x)}+\frac{1}{p'(x)}=1, \ \ x\in\Omega
$$
with the convention that $\frac{1}{\infty}=0$. The norm $\|\cdot\|_{L^{p(\cdot)}}$ satisfies the following norm conjugate formula (cf. \cite{DHHR11}, Corollary 3.2.14).
\begin{lemma}\label{le2.4}
Let $p(\cdot)\in \mathcal{P}(\Omega)$, and let $p'(\cdot)$ be the conjugate of $p(\cdot)$. Then we have
\begin{equation}\label{eq2.5}
   \frac{1}{2}\|f\|_{L^{p(\cdot)}}\leq \sup_{\|g\|_{L^{p'(\cdot)}}\leq1}\int_{\Omega}|f(x)g(x)|dx\leq 2\|f\|_{L^{p(\cdot)}}.
\end{equation}
\end{lemma}

We know that every function in the  variable Lebesgue space is locally integrable, and it holds the following embedding result with a sharper embedding constant (cf. \cite{CF13}, Corollary 2.48).
\begin{lemma}\label{le2.5}
Given $\Omega\subset\mathbb{R}^{n}$ and $p_{1}(\cdot), p_{2}(\cdot)\in \mathcal{P}(\Omega)$ with $1<p_{1}^{+}, p_{2}^{+}<+\infty$. Then  $L^{p_{2}(\cdot)}(\Omega)\subset L^{p_{1}(\cdot)}(\Omega)$ if and only if $p_{1}(x)\leq p_{2}(x)$ almost everywhere. Furthermore,  in this case we have that
\begin{equation}\label{eq2.6}
   \|f\|_{L^{p_{1}(\cdot)}}\leq (1+|\Omega|)\|f\|_{L^{p_{2}(\cdot)}}.
\end{equation}
\end{lemma}

An interesting fact in the setting of variable Lebesgue spaces is given by the extension of the result presented in Lemma \ref{le2.5} to unbounded domains. Before given a result in this direction let consider the following class of variable exponents.

\begin{definition}\label{de2.6}
Given a constant exponent
$p\in (1,\infty)$, we define the class of variable exponents $\mathcal{P}^{\text{emb}}_p(\mathbb{R}^2)$, as the set
\begin{equation*}
  \mathcal{P}^{\text{emb}}_p(\mathbb{R}^2) :=\left\{ \bar{p}(\cdot)\in\mathcal{P}^{\log}(\mathbb{R}^2):
  p\leq(\bar{p})^{-} \leq (\bar{p})^{+} <+\infty\ \ \text{ and } \ \
  \frac{ p\bar{p}(x)}{ \bar{p}(x)- p}\to +\infty \text{ as } |x|\to +\infty   \right\}.
\end{equation*}
\end{definition}

Then, a consequence of considering a variable exponent  in $\mathcal{P}^{\text{emb}}_p(\mathbb{R}^2)$
is given by following result   (cf. Theorem 2.45 and Remark 2.46 in \cite{CF13}).

\begin{lemma}\label{le2.7}
    Let $p\in (1,\infty)$ and  $\bar{p}(\cdot)\in\mathcal{P}^{\text{emb}}_p(\mathbb{R}^2)$. Then we have
    $L^{\bar{p}(\cdot)}(\mathbb{R}^2)\subset L^{p}(\mathbb{R}^2)$, and there exists a constant $C>0$ such that
\begin{equation}\label{eq2.7}
    \|f\|_{L^{p} (\mathbb{R}^2)} \leq C\|f\|_{L^{\bar{p}(\cdot)} (\mathbb{R}^2)}.
\end{equation}
\end{lemma}

\noindent\textbf{Remark 2.3}
We remark that not all properties of the usual Lebesgue spaces $L^{p}(\Omega)$ can be generalized to the variable Lebesgue spaces. For example, the convolution of two functions $f$ and $g$ is not well-adapted to the structure of the variable Lebesgue space $L^{p(\cdot)}(\Omega)$, and we know from \cite{CF13} that the Young's inequality are not valid anymore.   Furthermore, the Plancherel's formula is also failed in the variable Lebesgue space $L^{p(\cdot)}(\Omega)$, which is a very important tool to study PDEs.
\smallbreak

Now we back to the whole space $\mathbb{R}^{n}$. Note that the velocity field $u$ appeared in the equation \eqref{eq1.1} can be represented by the temperature $\theta$ via Riesz transforms $\mathcal{R}_{j}$ ($j=1,2$), which is bounded in the usual Lebesgue space $L^{p}(\mathbb{R}^{n})$ for all $1<p<\infty$. In order to study the boundedness of such singular integral operators in the variable Lebesgue spaces, a certain regularity has to be assumed on the variable exponent $p(\cdot)$: the so-called log-H\"{o}lder continuity of $p(\cdot)$. We introduce the following definition.

\begin{definition}\label{de2.8}
Let $p(\cdot)\in \mathcal{P}(\mathbb{R}^{n})$. We say $p(\cdot)\in \mathcal{P}^{\log}(\mathbb{R}^{n})$ if the following three conditions hold:
\begin{itemize}
\item The limit $\lim_{|x|\rightarrow\infty}\frac{1}{p(x)}$ exists, and we denote $\frac{1}{p_{\infty}}=\lim_{|x|\rightarrow\infty}\frac{1}{p(x)}$;
\item For all $x,y\in \mathbb{R}^{n}$, $\big{|}\frac{1}{p(x)}-\frac{1}{p(y)}\big{|}\leq \frac{C}{\log(e+\frac{1}{|x-y|})}$;
\item For all $x\in \mathbb{R}^{n}$, $\big{|}\frac{1}{p(x)}-\frac{1}{p_{\infty}}\big{|}\leq \frac{C}{\log(e+|x|)}$.
\end{itemize}
\end{definition}

For any $p(\cdot)\in \mathcal{P}^{\log}(\mathbb{R}^{n})$, we have the following results in terms of the Hardy-Littlewood maximal function and  Riesz transforms (cf. \cite{CF13}, Theorem 3.16 and Theorem 5.42;\cite{DHHR11}, Theorem 4.3.8 and Corollary 6.3.10).

\begin{lemma}\label{le2.9}
Let $p(\cdot)\in \mathcal{P}^{\log}(\mathbb{R}^{n})$ with $1<p^{-}\leq p^{+}<+\infty$. Then  for any $f\in L^{p(\cdot)}(\mathbb{R}^{n})$, there exists a positive constant $C$ such that
\begin{equation}\label{eq2.8}
   \|\mathcal{M} (f)\|_{L^{p(\cdot)}}\leq C\|f\|_{L^{p(\cdot)}},
\end{equation}
where $\mathcal{M}$ is the Hardy-Littlewood maximal function defined by
\begin{equation*}
  \mathcal{M}(f)(x):=\sup_{x\in B}\frac{1}{|B|}\int_{B}|f(y)|dy,
\end{equation*}
and $B\subset \mathbb{R}^{n}$ is an open ball with center $x$. Furthermore,
\begin{equation}\label{eq2.9}
   \|\mathcal{R}_{j}(f)\|_{L^{p(\cdot)}}\leq C\|f\|_{L^{p(\cdot)}} \ \ \text{for any}\ \ 1\leq j\leq n,
\end{equation}
where $\mathcal{R}_{j}$ ($1\leq j\leq n$) are the usual Riesz transforms.
\end{lemma}

We also need to use the boundedness of the Riesz potential in the variable Lebesgue spaces.

\begin{definition}\label{de2.10}
Given $0<\beta <n$, for any measurable function $f$, define the Riesz potential $\mathcal{I}_{\beta}$ also referred to as the fractional integral operator with index $\beta$, to be the convolution operator
\begin{equation}\label{eq2.10}
   \mathcal{I}_{\beta}(f)(x):=C(\beta, n)\int_{\mathbb{R}^{n}}\frac{|f(y)|}{|x-y|^{n-\beta}}dy,
\end{equation}
where
\begin{equation*}
   C(\beta, n)=\frac{\Gamma(\frac{n-\beta}{2})}{\pi^{\frac{n}{2}}2^{\beta}\Gamma(\frac{\beta}{2})}.
\end{equation*}
\end{definition}

The Riesz potential is well defined on the variable Lebesgue spaces, and if $p_{+}<\frac{n}{\beta}$ and $f\in L^{p(\cdot)}(\mathbb{R}^{n})$, then $\mathcal{I}_{\beta}(f)(x)$ converges for every $x$. Moreover, we have the following boundedness result (cf. \cite{CF13}, Theorem 5.46).

\begin{lemma}\label{le2.11}
Let $p(\cdot)\in \mathcal{P}^{\log}(\mathbb{R}^{n})$ with $1<p^{-}\leq p^{+}<+\infty$, and let $0<\beta<\frac{n}{p^{+}}$. Then for any $f\in L^{p(\cdot)}(\mathbb{R}^{n})$, there exists a positive constant $C$ such that
\begin{equation}\label{eq2.11}
   \|\mathcal{I}_{\beta}(f)\|_{L^{q(\cdot)}}\leq C\|f\|_{L^{p(\cdot)}} \ \ \text{with}\ \ \frac{1}{q(\cdot)}=\frac{1}{p(\cdot)}-\frac{\beta}{n}.
\end{equation}
\end{lemma}

In the end of this section, we recall the $L^{p}$- $L^{q}$ estimates of the fractional heat kernel $G^{\alpha}_{t}(x)$ (cf. \cite{MYZ08}, Lemma 3.1), where
\begin{equation*}
G^{\alpha}_{t}(x)=\mathcal{F}^{-1}(e^{-t|\xi|^{\alpha}})
=(2\pi)^{-\frac{n}{2}}\int_{\mathbb{R}^{n}}e^{ix\cdot\xi}e^{-t|\xi|^{\alpha}}d\xi.
\end{equation*}
\begin{lemma}\label{le2.12}
For all $\alpha>0$, $\nu>0$, $1\leq p\leq q\leq\infty$. Then for any $f\in L^{p}(\mathbb{R}^{n})$, we have
\begin{align}\label{eq2.12}
   &\|G^{\alpha}_{t}\ast f\|_{L^{q}}\leq Ct^{-\frac{n}{\alpha}(\frac{1}{p}-\frac{1}{q})}\|f\|_{L^{p}},\\
   \label{eq2.13}
      &\|\Lambda^{\nu}G^{\alpha}_{t}\ast f\|_{L^{q}}\leq Ct^{-\frac{\nu}{\alpha}-\frac{n}{\alpha}(\frac{1}{p}-\frac{1}{q})}\|f\|_{L^{p}}.
\end{align}
\end{lemma}

\section{The proof of Theorem \ref{th1.1}}
We shall show Theorem \ref{th1.1} by applying the following existence and uniqueness result for an abstract operator equation in a generic Banach space (cf. \cite{L02}, Theorem 13.2).
\begin{proposition}\label{pro3.1}
Let $\mathcal{X}$ be a Banach space and
$\mathcal{B}:\mathcal{X}\times\mathcal{X}\rightarrow\mathcal{X}$  a
bilinear bounded operator. Assume that for any $u,v\in
\mathcal{X}$, we have
$$
  \|\mathcal{B}(u,v)\|_{\mathcal{X}}\leq
  C\|u\|_{\mathcal{X}}\|v\|_{\mathcal{X}}.
$$
Then for any $y\in \mathcal{X}$ such that $\|y\|_{\mathcal{X}}\leq
\eta<\frac{1}{4C}$, the equation $u=y+\mathcal{B}(u,u)$ has a
solution $u$ in $\mathcal{X}$. Moreover, this solution is the only
one such that $\|u\|_{\mathcal{X}}\leq 2\eta$, and depends
continuously on $y$ in the following sense: if
$\|\widetilde{y}\|_{\mathcal{X}}\leq \eta$,
$\widetilde{u}=\widetilde{y}+\mathcal{B}(\widetilde{u},\widetilde{u})$
and $\|\widetilde{u}\|_{\mathcal{X}}\leq 2\eta$, then
$$
  \|u-\widetilde{u}\|_{\mathcal{X}}\leq \frac{1}{1-4\eta
  C}\|y-\widetilde{y}\|_{\mathcal{X}}.
$$
\end{proposition}

As a standard practice, we can reformulate the equation \eqref{eq1.1} into an equivalent integral equation by the Duhamel formula
\begin{equation}\label{eq3.1}
  \theta(t,x)=G^{\alpha}_{t}\ast \theta_{0}(x)+\int_{0}^{t}G^{\alpha}_{t-s}\ast f(s,x)ds-\int_{0}^{t}G^{\alpha}_{t-s}\ast(u\cdot\nabla\theta)(s,x)ds.
\end{equation}
For simplicity, we denote the bilinear term as
\begin{equation*}
  B(u,\theta):=\int_{0}^{t}G^{\alpha}_{t-s}\ast(u\cdot \nabla \theta)(s,x)ds.
\end{equation*}

Under the assumptions of Theorem \ref{th1.1}, we introduce the functional space $\mathcal{X}_{T}$ as
\begin{equation*}
   \mathcal{X}_{T}:=L^{q(\cdot)}(0,T; L^{p}(\mathbb{R}^2)),
\end{equation*}
where $T>0$ is a constant to be determined later. The space $\mathcal{X}_{T}$ is endowed with a Luxemburg norm as
\begin{equation*}
   \|f\|_{\mathcal{X}_{T}}=\inf\left\{\lambda>0: \int_{0}^{T}\left|\frac{\|\theta(t,\cdot)\|_{L^p}}{\lambda}\right|^{q(t)}dt\leq1\right\}.
\end{equation*}
Under this functional setting we will consider the Banach-Picard principle to construct mild solutions
for the integral equation \eqref{eq3.1}.
\smallbreak

We first establish the linear estimates of the integral equation \eqref{eq3.1} in the space $\mathcal{X}_{T}$.
\begin{lemma}\label{le3.2}
Let $\alpha\in(1,2]$, $q(\cdot)\in \mathcal{P}^{\log}(0,+\infty)$ with $2<q^{-}\leq q^{+}<+\infty$, fix an index $p>\frac{2}{\alpha-1}$ by the relationship $\frac{\alpha}{q(\cdot)}+\frac{2}{p}<\alpha-1$ and $\bar{p}(\cdot)\in\mathcal{P}^{\text{emb}}_{p}(\mathbb{R}^{2})$. Then for any $\theta_{0}\in L^{\bar{p}(\cdot)}(\mathbb{R}^{2})$,
there exists a positive constant $C$ such that
\begin{equation}\label{eq3.2}
   \|G^{\alpha}_{t}\ast \theta_{0}\|_{\mathcal{X}_{T}}\leq
   C\max\{T^{\frac{1}{q^{-}}}, T^{\frac{1}{q^{+}}}\}\|\theta_{0}\|_{L^{\bar{p}(\cdot)}_{x}}.
\end{equation}
\end{lemma}
\begin{proof}
First, we infer from Lemma \ref{le2.12} that
\begin{equation}\label{eq3.3}
   \|G^{\alpha}_{t}\ast \theta_{0}\|_{L^{p}_{x}}\leq
   \|G^{\alpha}_{t}(x)\|_{L^{1}_{x}}\|\theta_{0}\|_{L^{p}_{x}}=\|\theta_{0}\|_{L^{p}_{x}}.
\end{equation}
Then taking the $L^{q(\cdot)}$-norm to both sides of \eqref{eq3.3} with respect to the time variable $t$, and using the H\"{o}lder's inequality \eqref{eq2.3} yield that
\begin{equation}\label{eq3.4}
   \|G^{\alpha}_{t}\ast \theta_{0}\|_{L^{q(\cdot)}_{t}(L^{p}_{x})}\leq C\|1\|_{L^{q(\cdot)}_{t}}
   \|\theta_{0}\|_{L^{p}_{x}}\leq C\max\{T^{\frac{1}{q^{-}}}, T^{\frac{1}{q^{+}}}\}\|\theta_{0}\|_{L^{p}_{x}},
\end{equation}
where we used the following simple result in the variable Lebesgue spaces:
\begin{equation*}
   \|1\|_{L^{q(\cdot)}(0,T)}\leq
   2\max\{T^{\frac{1}{q^{-}}}, T^{\frac{1}{q^{+}}}\}.
\end{equation*}
Then by Lemma \ref{le2.7}, we conclude that
\begin{equation}\label{eq3.5}
   \|G^{\alpha}_{t}\ast \theta_{0}\|_{L^{q(\cdot)}_{t}(L^{p}_{x})}\leq  C\max\{T^{\frac{1}{q^{-}}}, T^{\frac{1}{q^{+}}}\}\|\theta_{0}\|_{L^{\bar{p}(\cdot)}_{x}}.
\end{equation}
We complete the proof of Lemma \ref{le3.2}.
\end{proof}
\smallbreak

\begin{lemma}\label{le3.3}
Let $\alpha\in(1,2]$, $q(\cdot)\in \mathcal{P}^{\log}(0,+\infty)$ with $2<q^{-}\leq q^{+}<+\infty$, fix an index $p>\frac{2}{\alpha-1}$ by the relationship $\frac{\alpha}{q(\cdot)}+\frac{2}{p}<\alpha-1$ and $\bar{p}(\cdot)\in\mathcal{P}^{\text{emb}}_{p}(\mathbb{R}^{2})$. Then for any $f\in L^{q(\cdot)}(0,+\infty; L^{\bar{p}(\cdot)}(\mathbb{R}^{2}))$, there exists a positive constant $C$ such that
\begin{equation}\label{eq3.6}
   \|\int_{0}^{t}G^{\alpha}_{t-s}\ast f(s,x)ds\|_{\mathcal{X}_{T}}\leq
   C\max\{T^{\frac{1}{q^{-}}}, T^{\frac{1}{q^{+}}}\}\|f\|_{L^{1}_{t}L^{\bar{p}(\cdot)}_{x}}.
\end{equation}
\end{lemma}

\begin{proof}
We start by the usual $L^q$-norm in the space variable and we obtain
\begin{equation*}
\left\|\int_{0}^{t}G^{\alpha}_{t-s}\ast f(s,x)ds\right\|_{L^{p}} \leq  \int_0^t\|G^{\alpha}_{t-s}(\cdot)\|_{L^1} \| f(s, \cdot) \|_{L^{p}}ds.
\end{equation*}
Then we can proceed the same lines as the proof of Lemma \ref{le3.2} to get
\begin{align}\label{le3.7}
\left\|\int_{0}^{t}G^{\alpha}_{t-s}\ast f(s,x)ds\right\|_{L^{q(\cdot)}_t(L^{p}_x)}
&\leq
C\left\|\| f \|_{L^1_t L^{\bar{p}(\cdot)}_x }\right\|_{L^{q(\cdot)}_t}\nonumber
\\
&\leq
 C\| f\|_{L^1_t L^{{\bar{p}(\cdot)}(\cdot) }_x}\|1\|_{L^{q(\cdot)}_{t}}\notag\nonumber
\\
&\leq
C\max\left\{T^{\frac{1}{q^{-}}}, T^{\frac{1}{q^{+}}}\right\}\| f\|_{L^1_t L^{\bar{p}(\cdot)  }_x}.
\end{align}
We complete the proof of Lemma \ref{le3.2}.
\end{proof}
\smallbreak

Finally, we establish the bilinear estimate of the integral equation \eqref{eq3.1} in the space $\mathcal{X}_{T}$.
\begin{lemma}\label{le3.4}
Let $\alpha\in(1,2]$, $q(\cdot)\in \mathcal{P}^{\log}(0,+\infty)$ with $2<q^{-}\leq q^{+}<+\infty$, and $\frac{2}{\alpha-1}<p<+\infty$ satisfying the relationship $\frac{\alpha}{q(\cdot)}+\frac{2}{p}<\alpha-1$. Then for any $T>0$,
there exists a positive constant $C$ such that
\begin{equation}\label{eq3.8}
   \|B(u,\theta)\|_{\mathcal{X}_{T}}\leq
   C (1+T)\|\theta\|_{\mathcal{X}_{T}}^{2}.
\end{equation}
\end{lemma}
\begin{proof}
For any $0<t\leq T$, taking $L^{p}$-norm to the bilinear term $B(u,\theta)$ with respect to the space variable $x$,  and using Lemma \ref{le2.12} and the Riesz transforms are bounded in the space $L^{p}$, we obtain
\begin{align}\label{eq3.9}
   \|B(u,\theta)\|_{L^{p}_{x}}&\leq \int_{0}^{t}\|\nabla G^{\alpha}_{t-s}\ast(u\theta)(s,\cdot)\|_{L^{p}_{x}}ds\nonumber\\
   &\leq \int_{0}^{t}\frac{1}{(t-s)^{\frac{1}{\alpha}+\frac{2}{\alpha p}}}\|u\|_{L^{p}_{x}}\|\theta\|_{L^{p}_{x}}ds\nonumber\\
   &\leq \int_{0}^{t}\frac{1}{(t-s)^{\frac{1}{\alpha}+\frac{2}{\alpha p}}}\|\theta\|_{L^{p}_{x}}^{2}ds.
\end{align}
Then taking $L^{q(\cdot)}$-norm with respect to the time variable $t$ and using the norm conjugate formula \eqref{eq2.5}, we see that
\begin{align}\label{eq3.10}
   \|B(u,\theta)\|_{L^{q(\cdot)}_{t}(L^{p}_{x})}&\leq \left\|\int_{0}^{t}\frac{1}{(t-s)^{\frac{1}{\alpha}+\frac{2}{\alpha p}}}\|\theta\|_{L^{p}_{x}}^{2}ds\right\|_{L^{q(\cdot)}_{t}}\nonumber\\
   &\leq 2\sup_{\|\psi\|_{L^{q'(\cdot)}_{t}}\leq 1}\int_{0}^{T}\int_{0}^{t}\frac{|\psi(t)|}{|t-s|^{\frac{1}{\alpha}+\frac{2}{\alpha p}}}\|\theta(s,\cdot)\|_{L^{p}_{x}}^{2}dsdt\nonumber\\
   &=2\sup_{\|\psi\|_{L^{q'(\cdot)}_{t}}\leq 1}\int_{0}^{T}\int_{0}^{T}\frac{1_{\{0<s<t\}}|\psi(t)|}{|t-s|^{\frac{1}{\alpha}+\frac{2}{\alpha p}}}dt\|\theta(s,\cdot)\|_{L^{p}_{x}}^{2}ds.
\end{align}
In order to use the 1D Riesz potential formula \eqref{eq2.10}, we extend the function $\psi(t)$ by zero on $\mathbb{R}\setminus[0,T]$, we see that the right-hand side of \eqref{eq3.10} can be represented as
\begin{align}\label{eq3.11}
   &\sup_{\|\psi\|_{L^{q'(\cdot)}_{t}}\leq 1}\int_{0}^{T}\int_{0}^{T}\frac{1_{0<s<t}|\psi(t)|}{|t-s|^{\frac{1}{\alpha}+\frac{2}{\alpha p}}}dt\|\theta(s,\cdot)\|_{L^{p}_{x}}^{2}ds\nonumber\\
   &=\sup_{\|\psi\|_{L^{q'(\cdot)}_{t}}\leq 1}\int_{0}^{T}\left(\int_{-\infty}^{+\infty}\frac{|\psi(t)|}{|t-s|^{\frac{1}{\alpha}+\frac{2}{\alpha p}}}dt\right)\|\theta(s,\cdot)\|_{L^{p}_{x}}^{2}ds\nonumber\\
   &=\sup_{\|\psi\|_{L^{q'(\cdot)}_{t}}\leq 1}\int_{0}^{T}\mathcal{I}_{\beta}(|\psi|)\|\theta\|_{L^{p}_{x}}^{2}ds,
\end{align}
where $\beta=1-\frac{1}{\alpha}-\frac{2}{\alpha p}$. Furthermore, for the right-hand side of \eqref{eq3.11},
by using H\"{o}lder's inequality with $\frac{2}{q(x)}+\frac{1}{\widetilde{q}(x)}=1$ yields that
\begin{align}\label{eq3.12}
   \sup_{\|\psi\|_{L^{q'(\cdot)}_{t}}\leq 1}\int_{0}^{T}\mathcal{I}_{\beta}(|\psi|)\|\theta\|_{L^{p}_{x}}^{2}ds&\leq C\sup_{\|\psi\|_{L^{q'(\cdot)}_{t}}\leq 1}\|\mathcal{I}_{\beta}(|\psi|)\|_{L^{\widetilde{q}(\cdot)}_{t}}\|\theta\|_{L^{q(\cdot)}_{t}(L^{p}_{x})}^{2}\nonumber\\
   &\leq C\sup_{\|\psi\|_{L^{q'(\cdot)}_{t}}\leq 1}\|\psi\|_{L^{r(\cdot)}_{t}}\|\theta\|_{L^{q(\cdot)}_{t}(L^{p}_{x})}^{2},
\end{align}
where the above indices satisfy the relationship
$$
\frac{1}{\widetilde{q}(\cdot)}=\frac{1}{r(\cdot)}-(1-\frac{1}{\alpha}-\frac{2}{\alpha p}).
$$
 Since $\frac{1}{\widetilde{q}(x)}=1-\frac{2}{q(x)}$ and $\frac{1}{q'(x)}=1-\frac{1}{q(x)}$, we can deduce that $r(\cdot)<q'(\cdot)$ under the condition $\frac{\alpha}{q(\cdot)}+\frac{2}{p}<\alpha-1$. By using Lemma \ref{le2.5} with $r(\cdot)<q'(\cdot)$ and $\Omega=[0,T]$, we get
\begin{align}\label{eq3.13}
   \sup_{\|\psi\|_{L^{q'(\cdot)}_{t}}\leq 1}\|\psi\|_{L^{r(\cdot)}_{t}}\|\theta\|_{L^{q(\cdot)}_{t}(L^{p}_{x})}^{2}&\leq \sup_{\|\psi\|_{L^{q'(\cdot)}_{t}}\leq 1}\|\psi\|_{L^{q'(\cdot)}_{t}}\|\theta\|_{L^{q(\cdot)}_{t}(L^{p}_{x})}^{2}\nonumber\\
   &\leq C(1+T)\|\theta\|_{L^{q(\cdot)}_{t}(L^{p}_{x})}^{2}.
\end{align}
Finally, taking all above estimates \eqref{eq3.11}--\eqref{eq3.13} into \eqref{eq3.10}, we get \eqref{eq3.8}. We complete the proof of Lemma \ref{le3.4}.
\end{proof}
\smallbreak

Based on the desired linear and bilinear estimates obtained in Lemmas \ref{le3.2}, \ref{le3.3} and \ref{le3.4}, we know that,  for any $0<T<\infty$,  there exist two positive
constants $C_{1}$ and $C_{2}$ such that
\begin{align}\label{eq3.14}
  \|\theta\|_{\mathcal{X}_{T}}\leq
 C_{1}\max\{T^{\frac{1}{q^{-}}}, T^{\frac{1}{q^{+}}}\}(\|\theta_{0}\|_{L^{\bar{p}(\cdot)}_{x}}+\|f\|_{L^{1}_{t}L^{\bar{p}(\cdot)}_{x}})
  +C_{2}  (1+T)\|\theta\|_{\mathcal{X}_{T}}^{2}.
\end{align}
Hence, if we choose $T$ small enough such that
\begin{align*}
\|\theta_{0}\|_{L^{\bar{p}(\cdot)}_{x}}+\|f\|_{L^{1}_{t}L^{\bar{p}(\cdot)}_{x}}\leq \frac{1}{4C_{1}C_{2} (1+T)\max\{T^{\frac{1}{q^{-}}}, T^{\frac{1}{q^{+}}}\}},
\end{align*}
we know from Proposition \ref{pro3.1} that the equation \eqref{eq1.1} admits a unique solution $\theta\in \mathcal{X}_{T}$. The proof of Theorem \ref{th1.1} is achieved.

\section{The proof of Theorem \ref{th1.2}}
We begin by considering a multi-index $\gamma$ such that $|\gamma|=1$. Taking $D$ on the integral equation \eqref{eq3.1}, we can write

\begin{equation}\label{eq4.1}
  D\theta(t,x)=G^{\alpha}_{t}\ast D\theta_{0}(x)+\int_{0}^{t}G^{\alpha}_{t-s}\ast Df(s,x)ds- B(Du,\theta)- B(u,D\theta).
\end{equation}
Thus, in the following we will obtain an unique solution of \eqref{eq3.1} by
considering the Banach-Picard principle in the space  $\mathcal{X}_T= L^{q(\cdot)} ( 0,T; L^{p} (\mathbb{R}^2))$.
Thus, under the assumptions \eqref{eq1.5}, we can follow the similar arguments as in Lemmas \ref{le3.2} and  \ref{le3.3} to obtain that
\begin{equation}\label{eq4.2}
\|G^{\alpha}_{t}\ast D\theta_{0}\|_{\mathcal{X}_T}\leq C
\max\left\{T^{\frac{1}{q^{-}}}, T^{\frac{1}{q^{+}}}\right\}
 \|D\theta_0\|_{L^{p}_{x}}
\end{equation}
and
\begin{equation}\label{eq4.3}
\left\|\int_{0}^{t}G^{\alpha}_{t-s}\ast Df(s,x)ds\right\|_{\mathcal{X}_T}
\leq C \max\left\{T^{\frac{1}{q^{-}}}, T^{\frac{1}{q^+} } \right\}\| Df \|_{L^1_t(L^p_x)}.
\end{equation}
Moreover, in order to deal with the nonlinearities we can also follow the similar arguments as in Lemma \ref{le3.4} to get
 \begin{align}\label{eq4.4}
        &\| B(Du,\theta)\|_{\mathcal{X}_T} \leq C (1+T)\|Du\|_{\mathcal{X}_T}\|\theta\|_{\mathcal{X}_T}\leq C (1+T)\|D\theta\|_{\mathcal{X}_T}\|\theta\|_{\mathcal{X}_T};\\
        \label{eq4.5}
          &\| B(u,D\theta)\|_{\mathcal{X}_T} \leq C (1+T)\|u\|_{\mathcal{X}_T}\|D\theta\|_{\mathcal{X}_T}\leq C (1+T)\|\theta\|_{\mathcal{X}_T}\|D\theta\|_{\mathcal{X}_T}.
 \end{align}
Putting the above estimates \eqref{eq4.2}--\eqref{eq4.5} together,  we conclude that there exist $0<T<+\infty$ such that
if the initial data $\theta_{0}$ and the external force $f$ satisfy the condition:
 \begin{align*}
 \|D\theta_0\|_{L^{p}_{x}}+\|Df\|_{L^1_t(L^{p}_x)}\leq\frac{C}{(1+T)\max\left\{T^{\frac{1}{q^{-}}}, T^{\frac{1}{q^{+}}}\right\}},
 \end{align*}
we can apply Proposition \ref{pro3.1} to obtain an unique solution
$D \theta$ of the integral system \eqref{eq3.1}.
The uniqueness of the derivative implies that this solution coincide with the first-order derivative applied to the unique solution of Theorem  \ref{th1.1}.
By iterating this process we easily obtain that $D^\gamma\theta \in L^{q(\cdot)} \left( 0,T;L^{p} (\mathbb{R}^2) \right)$ for $|\gamma|=2,3,...,|\beta|$.
With this we conclude the proof of Theorem \ref{th1.2}.

\bigskip

\noindent \textbf{Acknowledgements.}
The authors declared that they have no conflict of interest. G. Vergara-Hermosilla was supported by the ANID postdoctoral program BCH 2022 (no. 74220003), also he thanks Pierre-Gilles Lemari\'{e}-Rieusset and Diego Chamorro for their helpful comments and advises.
J. Zhao was partially supported by the National Natural Science Foundation of China (no. 12361034) and
 the Natural Science Foundation of Shaanxi Province (no. 2022JM-034).



\begin{thebibliography}{100}
\setlength{\itemsep}{-2mm}
\bibitem{AH08} H. Abidi, T. Hmidi,  On the global well-posedness of the
critical quasi-geostrophic equation,  SIAM J. Math. Anal.   40 (2008) 167--185.

\bibitem{B08} H. Bae, Global well-posedness of dissipative quasi-geostrophic equations in the critical spaces, Proc. Amer. Math. Soc.
136 (2008) 257--262.

\bibitem{BBT13} H. Bae, A. Biswas, E. Tadmor, Analyticity of the subcritical and critical quasi-geostrophic equations in
Besov spaces, arXiv.1310.1624v1.

\bibitem{BB15} J. Benameura,  M. Benhamedb, Global existence of the two-dimensional QGE with sub-critical dissipation, J. Math. Anal. Appl. 423 (2015) 1330--1347.



\bibitem{CF07} J.A Carrillo, L.C.F. Ferreira, Self similar solutions and large time asymptotics for the dissipative
quasi-geostrophic equations, Monatsh. Math. 151 (2007) 111--142.

\bibitem{CF08} J.A Carrillo, L.C.F. Ferreira, The asymptotic behavior of subcritical dissipative quasi-geostrophic equations, Nonlinearity 21 (2008) 1001--1018.

\bibitem{CL03} D. Chae, J. Lee,  Global well-posedness in the super-critical
dissipative quasigeostrophic equations,  Comm. Math. Phys. 233 (2003) 297--311.
		

\bibitem{CV23} D. Chamorro, G. Vergara-Hermosilla, Lebesgue spaces with variable exponent: some applications to the Navier-Stokes equations, Positivity 28(2) 2024 p. 24. 


\bibitem{CMZ07} Q. Chen, C. Miao,  Z. Zhang,  A new Bernstein's inequality and the 2D dissipative quasi-geostrophic equation, Comm. Math. Phys.
271 (2007) 821--838.

\bibitem{CZ07}  Q. Chen,  Z. Zhang, Global well-posedness of the 2D critical
dissipative quasi-geostrophic equation in the Triebel-Lizorkin spaces,  Nonlinear Anal.  67(2007) 1715--1725.

\bibitem{CCW01} P. Constantin, D. C\'ordoba, J. Wu,  On the critical dissipative quasi-geostrophic equation, Indiana Univ. Math. J.
50 (2001) 97--107.

\bibitem{CMT94} P. Constantin, A.J. Majda, E. Tabak, Formation of strong fronts in the 2D quasi-geostrophic thermal active scalar,
Nonlinearity 7 (1994) 1495--1533.





\bibitem{CF13} D.V. Cruz-Uribe, A. Fiorenza, \textit{Variable Lebesgue Spaces: Foundations and Harmonic Analysis}, Birkh\"{a}user/Springer, Heidelberg, 2013.



\bibitem{DD08}  H. Dong, D. Du, Global well-posedness and a decay estimate for the critical dissipative quasi-geostrophic equation in the whole space, Discrete Contin. Dyn. Syst. 21(4) (2008) 1095--1101.


\bibitem{DL10} H. Dong, D. Li, On the 2D critical and supercritical dissipative quasi-geostrophic equation in Besov spaces, J. Differential
Equations 248(11) (2010) 2684--2702.

\bibitem{DHHR11}  L. Diening, P. Harjulehto, P. H\"{a}st\"{o}, M. Ruzicka, \textit{Lebesgue and Sobolev Spaces with
Variable Exponents}, Lecture Notes in Mathematics 2017, Springer, Heidelberg, 2011.



\bibitem{HK07} T. Hmidi, S. Keraani, Global solutions of the super-critical 2D quasi-geostrophic equation in Besov spaces, Adv.
Math. 214(2) (2007) 618--638.

\bibitem{J04} N. Ju, Existence and uniqueness of the solution to the dissipative 2D quasi-geostrophic equations in the Sobolev
space, Comm. Math. Phys. 251(2) (2004)  365--376.





\bibitem{L13}  O. Lazer, Global existence for the critical dissipative surface
quasi-geostrophic equation,  Comm. Math. Phys.  322 (2013) 73--93.

\bibitem{L02} P.-G. Lemari\'{e}-Rieusset, \textit{Recent Developments in the Navier-Stokes Problem}, Research
Notes in Mathematics, Chapman \& Hall/CRC, 2002.




\bibitem{MYZ08} C. Miao, B. Yuan, B. Zhang, Well-posedness of the Cauchy problem for the fractional
power dissipative equations, Nonlinear Anal. 68 (2008) 461--484.

\bibitem{M06} H. Miura, Dissipative quasi-geostrophic equation for large initial data in the critical Sobolev space, Comm. Math. Phys. 267(1) (2006) 141--157.


\bibitem{P87} J. Pedlosky, \textit{Geophysical Fluid Dynamics}, Springer, New York, 1987.



\bibitem{V24} G. Vergara-Hermosilla, Remarks on variable Lebesgue spaces and fractional Navier-Stokes equations, arXiv:2402.07508v1.


\bibitem{W06} J. Wu, Lower bounds for an integral involving fractional Laplacians and the generalized Navier-Stokes equations in
Besov spaces, Comm. Math. Phys. 263(3) (2006) 803--831.

\bibitem{W07} J. Wu,  Existence and uniqueness results for the 2-D dissipative quasi-geostrophic equation, Nonlinear Anal. 67(11) (2007)
 3013--3036.

\end{thebibliography}
\end{document}